\documentclass[a4paper,10pt]{article}

\usepackage{amsthm}

\usepackage{amsmath}

\usepackage{amsfonts}

\usepackage{amssymb}

\usepackage{amsmath,amsthm,amscd,amssymb}

\usepackage{latexsym}

\usepackage{graphicx}

\usepackage[all]{xy}

\theoremstyle{plain}

\newtheorem{thm}{\textbf{Theorem}}

\newtheorem{lem}{\textbf{Lemma}}

\newtheorem{cor}{\textbf{Corollary}}

\newtheorem{prop}{\textbf{Proposition}}

\newcommand{\A}{\Bbb{A}}

\newcommand{\R}{\Bbb{R}}

\newcommand{\Q}{\Bbb{Q}}

\newcommand{\F}{\Bbb{F}}

\newcommand{\Z}{\Bbb{Z}}

\newcommand{\T}{\Bbb{T}}

\newcommand{\p}{\frak{p}}

\newcommand{\vv}{\tilde{v}}

\newcommand{\m}{\mathfrak{m}}

\newcommand{\K}{\mathcal{K}}

\newcommand{\Frob}{\text{Frob}}

\newcommand{\Res}{\text{Res}}

\newcommand{\Hom}{\text{Hom}}

\newcommand{\Ann}{\text{Ann}}

\newcommand{\im}{\text{im}}

\newcommand{\Gal}{\text{Gal}}

\newcommand{\GL}{\text{GL}}

\newcommand{\Spec}{\text{Spec}}

\newcommand{\Norm}{\text{Norm}}

\newcommand{\soc}{\text{soc}}

\newcommand{\y}{\hspace{6pt}}

\newcommand{\mb}{\mathbb}
\newcommand{\mc}{\mathcal}
\newcommand{\mf}{\mathfrak}

\title{{\bf{Strong local-global compatibility in the $p$-adic Langlands program for $U(2)$}}}

\author{Przemyslaw Chojecki $\y$ Claus Sorensen}

\begin{document}

\date{\today}

\maketitle

\begin{abstract}
For certain mod $p$ Galois representations $\bar{\rho}$, arising from modular forms on definite unitary groups in two variables, we express the $\bar{\rho}$-part of completed cohomology $\hat{H}_{\bar{\rho}}^0$ (away from $\Sigma=\Sigma_p\cup \Sigma_0$) as a tensor product $\Pi_p\otimes \Pi_{\Sigma_0}$. Here $\Pi_p$ is attached to the universal deformation $\rho^{univ}$ via the $p$-adic local Langlands correspondence for $\GL_2(\Q_p)$, and $\Pi_{\Sigma_0}$ is given by the local Langlands correspondence in families, of Emerton and Helm. 

\footnote{{\it{Keywords}}: Galois representations, automorphic forms, $p$-adic Langlands program}
\footnote{{\it{2000 AMS Mathematics Classification}}: 11F33.}

\end{abstract}


\section{Introduction}

We continue the study from \cite{cs} of the completed cohomology $\hat{H}^0$ of definite unitary groups in two variables; or more precisely, the piece associated with a modular mod $p$ Galois representation $\bar{\rho}$, which is irreducible at $p$. The goal of this paper is to explicitly describe $\hat{H}_{\bar{\rho}}^0$ as a tensor product of representations arising from $p$-adic local Langlands for $\GL_2(\Q_p)$ and local Langlands in families (defined away from $p$) respectively. 

\medskip

\noindent Our unitary groups are defined relative to a fixed CM-extension $\K/F$. Thus $U=U(D,\star)_{/F}$, where $D$ is a quaternion $\K$-algebra with involution $\star$. We will always assume $G=\Res_{F/\Q}(U)$ is compact at infinity, and split at $p$. Here $p>2$ is a fixed prime, which splits in $\K$. Thus $G(\R)$ is a product of copies of $U(2)$, and $G(\Q_p)$ is a product of copies of $\GL_2(\Q_p)$. Let $\Sigma_p$ denote the set of places $v$ of $F$ above $p$. Choose a finite set of banal places $\Sigma_0$, disjoint from $\Sigma_p$, and let $\Sigma=\Sigma_0\cup \Sigma_p$. We assume $U$ splits at all these places, and for each $v \in \Sigma$ we once and for all choose a divisor $\vv|v$ in $\K$. Via this selection, we identify $U(F_v)\simeq \GL_2(\K_{\vv})$ -- which is well-defined up to conjugacy.

\medskip

\noindent Throughout, we fix a "large enough" coefficient field -- a finite extension $E/\Q_p$, with integers $\mc{O}=\mc{O}_E$, uniformizer $\varpi=\varpi_E$, and residue field $k=k_E$. We start off with an (absolutely) irreducible mod $p$ Galois representation,
$$
\bar{\rho}:\Gamma_{\K}=\Gal(\bar{\K}/\K)\longrightarrow \GL_2(k), 
$$
which we assume to be modular. That is, $\bar{\rho}\simeq \bar{\rho}_{\pi}$, for some automorphic representation $\pi$ of $G(\A)$ of weight $V$ and level $K=\prod_v K_v$. Throughout, we impose the following conditions on $\bar{\rho}$:

\begin{itemize}
\item $\bar{\rho}(\Gamma_{F(\zeta_p)})$ is big (in a formal sense).
\item $\bar{\rho}|_{\Gamma_{\K_{\vv}}}$ is irreducible, for all $v|p$.
\item $K=K_{max}\subset G(\A_f)$ is maximal, and hyperspecial outside $\Sigma$.
\item $V$ is in the Fontaine-Laffaille range $[0,p-2)$, at each infinite place.
\end{itemize}

\noindent (All but the second one can be traced back to the setup in \cite{CHT}.)

\medskip

\noindent Let $\m=\m_{\bar{\rho}}$ be the associated maximal ideal in the Hecke algebra. Here we are a bit imprecise about which Hecke algebra; we pull-back $\m$ via surjections of Hecke algebras without mention. Usually we view $\m$ as an ideal of $\T_{\Sigma}$, the quotient of the polynomial algebra $\T_{\Sigma}^{abs}$, in variables $T_w^{(1)}$ and $T_w^{(2)}$, which acts faithfully on the completed cohomology,
$$
\text{$\hat{H}^0(K^{\Sigma})_{\mc{O}}=\varinjlim_{K_{\Sigma_0}} \hat{H}^0(K_{\Sigma_0}K^{\Sigma})_{\mc{O}}$, $\y$
$\hat{H}^0(K^p)_{\mc{O}}=(\varinjlim_{K_p}H^0(K_pK^p)_{\mc{O}})^{\wedge}$.}
$$ 
We let $\rho_{\m}:\Gamma_{\K}\rightarrow \GL_2(\T_{\Sigma,\m})$ be the universal modular deformation of $\bar{\rho}$ over the "big" localized Hecke algebra $\T_{\Sigma,\m}$, which acts faithfully on $\hat{H}^0(K^{\Sigma})_{\mc{O},\m}$. This is our main object of interest. Our goal in this paper is to factor the $U(F_{\Sigma})$-action as a tensor product of a $G(\Q_p)$-representation $\Pi_p$, and a $U(F_{\Sigma_0})$-representation $\Pi_{\Sigma_0}$ (both over $\T_{\Sigma,\m}$). 

\medskip

\noindent At each place $v|p$, consider the restriction $\bar{\rho}|_{\Gamma_{\K_{\vv}}}$, which is irreducible by assumption -- and in particular in the image of Colmez's Montreal functor. (Recall that $\K_{\vv}=\Q_p$, so this makes sense.) Thus $\bar{\rho}|_{\Gamma_{\K_{\vv}}}\leftrightarrow \bar{\pi}_{\vv}$ for some $k$-representation $\bar{\pi}_{\vv}$ of $\GL_2(\K_{\vv})$. We refer to section 4.4 in \cite{cs} for a review of the salient facts. The correspondence pairs deformations, so $\rho_{\m,\vv}$ corresponds to some deformation $\Pi_{\m,\vv}$ of $\bar{\pi}_{\vv}$ over $\T_{\Sigma,\m}$. We let $\Pi_p=\hat{\otimes}_{v|p}\Pi_{\m,\vv}$, viewed as a representation of $G(\Q_p)$ over $\T_{\Sigma,\m}$. Here $\hat{\otimes}$ denotes the $\varpi$-adically completed tensor product.

\medskip

\noindent At each place $v \in \Sigma_0$, the restriction $\rho_{\m,\vv}: \Gamma_{\K_{\vv}}\rightarrow \GL_2(\T_{\Sigma,\m})$ yields a corresponding representation 
$\Pi_{\m,\vv}$ of $\GL_2(\K_{\vv})$ over $\T_{\Sigma,\m}$ via local Langlands in families -- as developed by Emerton and Helm in \cite{eh} and \cite{He} (see 2.3 for a quick summary of its defining properties). The normalization here is that $\Pi_{\m,\vv}$ is {\it{cotorsion}}-free (as opposed to torsion-free), and thus its smooth {\it{dual}} is the representation which interpolates the local Langlands correspondence, in the natural sense. We let $\Pi_{\Sigma_0}=(\otimes_{v \in \Sigma_0} \Pi_{\m,\vv})_{ctf}$; the largest cotorsion-free submodule of the tensor product (cf. Proposition 6.4.2 in \cite{eh}). As always, via the selection of divisors $\{\vv\}_{v \in \Sigma_0}$, we view $\Pi_{\Sigma_0}$ as a representation of
$U(F_{\Sigma_0})$.

\medskip

\noindent The main result we present in this note is the following.

\begin{thm}
With notation and assumptions as above, there is a $G(\Q_p)\times U(F_{\Sigma_0})$-equivariant isomorphism of $\T_{\Sigma,\m}$-modules,
$$
\Pi_p \overset{\curlywedge}{\otimes}_{\T_{\Sigma,\m}} \Pi_{\Sigma_0}\overset{\sim}{\longrightarrow} \widehat{H}^0(K^{\Sigma})_{\mc{O},\m}.
$$
(Here $\overset{\curlywedge}{\otimes}_{\T_{\Sigma,\m}}$ denotes Emerton's completed tensor product -- see section 3 below.)
\end{thm}

\noindent This result is formally similar to the (main) local-global compatibility conjecture 6.1.6 in \cite{em1} for $\hat{H}^1$ of the tower of modular curves -- which Emerton proves for $p$-distinguished $\bar{\rho}$ in Theorem 6.2.13 of loc. cit. (i.e., for $\bar{\rho}$ such that $\bar{\rho}|_{\Gamma_{\Q_p}}^{ss}$ is not of the form $\chi\oplus \chi$). 

\medskip

\noindent In our setup, the arithmetic manifolds are finite sets, and we only have cohomology in degree zero. Furthermore, $\hat{H}^0$ is much simpler -- it can be realized as continuous functions on a compact set, and there is no Galois-action. That being said, our approach relies {\it{heavily}} on that of Emerton -- and for the most part our proof is an almost "formal" application of his general results on coadmissible modules in his Appendix C. However, there seems to be one step in \cite{em1} which does not straight-forwardly adapt to the unitary setting. Namely, how Ihara's lemma ensures that a certain candidate representation for $(\Pi_{\Sigma_0}/\varpi \Pi_{\Sigma_0})[\m]$ is indeed {\it{essentially AIG}} (a key notion from \cite{eh}). In the companion paper \cite{sor2}, this was taken care of -- even for $U(n)$, where Ihara's "lemma" is still a big open problem for $n>2$. This led to the writing of this note, which we hope will be of some interest -- and of some expository value.

\medskip

\noindent Note that in the special case where $\Sigma_0=\varnothing$, Theorem 1 realizes $\Pi_p$ globally as
$\widehat{H}^0(K^p)_{\mc{O},\m}$, with $K^p$ being a product of hyperspecials away from $p$. This suggests that the elusive $p$-adic local Langlands correspondence for 
$n$-dimensional $\rho$ should be closely related to the analogous piece of completed cohomology for $U(n)$ -- at least when $\bar{\rho}$ satisfies the criteria given earlier.

\section{Notation and recollections}

\subsection{Definite unitary groups and their Hecke algebras}

\noindent We briefly recall the setting from our paper \cite{cs}. 

\medskip

\noindent The unitary groups we work with are defined relative to a CM-extension $\K/F$. More precisely, we let $D$ be a quaternion algebra over $\K$, endowed with an $F$-linear anti-involution $\star$ of the second kind ($\star|_{\K}=c$). This pair defines a unitary group $U=U(D,\star)_{/F}$, an inner form of $\GL(2)$ over $\K$. Indeed, $U \times_F \K \simeq D^{\times}$. We consider the restriction of scalars, $G=\text{Res}_{F/\Q}(U)$. We always assume $G(\R)$ is {\it{compact}} and that $D$ splits at all places of $\K$ which lie above a fixed prime number $p$, which splits in $\K$.

\medskip

\noindent For each place $v|p$ of $F$, we choose a place $\vv|v$ of $\K$ above it (note that $\K_{\vv}=\Q_p$ canonically). Using this selection of places $\{\vv\}$, we have identifications
$$
\text{$G(\R)\overset{\sim}{\longrightarrow} U(2)^{\Hom(F,\R)}$, $\y$ $G(\Q_p) \overset{\sim}{\longrightarrow} \prod_{v|p} \GL_2(\K_{\vv})$.}
$$
(Of course, $\K_{\vv}=\Q_p$, but we wish to incorporate $\vv$ in our notation to emphasize how our identification depends on this choice. Hence we stick to the somewhat cumbersome notation $\K_{\vv}$.)







\medskip

\noindent Throughout, we fix a finite set $\Sigma _0$ of finite places of $F$ -- none of which lie above $p$. We assume each $v \in \Sigma_0$ splits in $\K$, and that $D$ splits at every place above such a $v$. Let $\Sigma = \Sigma _0 \cup \Sigma_p$, where $\Sigma_p=\{ v|p \}$ consists of all places above $p$. We enlarge our selection $\{\vv\}$, and choose a divisor $\vv$ of each $v \in \Sigma$.

\medskip

\noindent For any compact open subgroup $K \subset G(\mb{A} _f)$ we consider the arithmetic manifold 
$$
Y(K) = G(\Q) \backslash G(\mb{A}_f) / K
$$
which is just a finite set. For any commutative ring $A$ we denote by $H^0(K)_A$ the set of functions $Y(K) \rightarrow A$. For each tame level $K^p \subset G(\mb{A}_f^p)$, one defines (following Emerton) the {\it{completed}} cohomology, with coefficients in some $p$-adic ring of integers $\mc{O}=\mc{O}_E$, having uniformizer $\varpi$,
$$
\hat{H} ^0(K^p) _{\mc{O}} = \varprojlim _{s} \varinjlim _{K_p} H^0(K_pK^p) _{\mc{O} / \varpi ^s \mc{O}},
$$
where the direct limit runs over compact open subgroups $K_p \subset G(\Q _p)$, and $s>0$. This is the unit ball in the $p$-adic Banach space 
$\hat{H} ^0(K^p) _{E}$, on which $U(F_{\Sigma_p})\simeq \prod_{v|p}\GL_2(\K_{\vv})$ acts unitarily. Next, we factor $K^p=K_{\Sigma_0}K^{\Sigma}$ and let $K_{\Sigma_0}$ shrink to the identity. We define
$$
\hat{H} ^0(K^{\Sigma}) _{\mc{O}}= \hat{H} ^0 _{\mc{O}, \Sigma} = \varinjlim _{K_{\Sigma_0}} H^0(K_{\Sigma_0}K^{\Sigma}) _{\mc{O}},
$$
which by definition carries a smooth $U(F_{\Sigma_0})$-action. In what follows, we will always take $K^{\Sigma}$ to be a product of hyperspecial maximal compact subgroups $K_v=U(\mc{O}_v)$, for $v \notin \Sigma$. (Note that there are {\it{two}} conjugacy classes of such when $U_{/F_v}$ is a $p$-adic unitary group.)

\medskip

\noindent We denote by $\mb{T} _{\Sigma}^{abs}$ the polynomial $\mc{O}$-algebra in the variables $T_w ^{(1)}$ and $T_w ^{(2)}$, one for each place $w$ of $\K$ lying above $v=w|_F\notin \Sigma$, which splits in $\K$. Here the superscript {\it{abs}} stands for "abstract" to remind us that this is a Hecke algebra of infinite type. Once and for all, we choose an isomorphism $\iota_w: U(F_v) \overset{\sim}{\longrightarrow} \GL_2(\K_w)$ in such a way that it identifies $U(\mc{O}_v)\simeq \GL_2(\mc{O}_{\K_w})$.

\medskip

\noindent The algebra $\mb{T} _{\Sigma}^{abs}$ acts naturally on the module $\hat{H}_{\mc{O},\Sigma}^0$. Explicitly, $T_w^{(j)}$ acts via the usual double coset operator
$$
\iota _w ^{-1} \left[ \GL _2 (\mc{O} _{\K_w}) \begin{pmatrix} 1 _{2-j} & \ \\ \ & \varpi _w 1 _j \end{pmatrix} \GL_2(\mc{O} _{\K_w}) \right],
$$
where $\varpi _w$ is a choice of uniformizer in $\mc{O} _{\K_w}$. We let $\mb{T}_{\Sigma}$ denote the quotient of $\mb{T} _{\Sigma}^{abs}$ which acts faithfully on 
$\hat{H}_{\mc{O},\Sigma}^0$, and we may (tacitly) view $T_w^{(j)}$ as an operator on $\hat{H}_{\mc{O},\Sigma}^0$ or a variable in $\mb{T} _{\Sigma}^{abs}$, interchangeably.
For a compact open subgroup $K_{\Sigma_0}$ as above (a "level") we let $\mb{T}_{\Sigma}(K_{\Sigma_0})$ be the quotient of $\mb{T}_{\Sigma}$ cut out by 
the submodule of $K_{\Sigma_0}$-invariants, $\hat{H}^0(K_{\Sigma_0}K^{\Sigma}) _{\mc{O}}$.

\medskip

\noindent For each maximal ideal $\m \subset \mb{T} _{\Sigma}^{abs}$, say with residue field $k=\mb{T} _{\Sigma}^{abs}/\m$, which is a finite extension of $k_E\supset \F_p$, 
we may consider the localization $\hat{H} ^0(K^{\Sigma}) _{\mc{O},\m}$, which will be our main object of study. Those maximal ideals we will eventually look at arise from a "modular" Galois representation, $\bar{\rho} : \Gamma _{\K} \rightarrow \GL_2 (k_E)$, which is unramified outside $\Sigma$, as follows: $\m=\m_{\bar{\rho}} \subset \mb{T} _{\Sigma}^{abs}$ is generated by $\varpi=\varpi_E$ and all the elements
$$ 
((-1)^j \Norm (w) ^{j(j-1)/2} T_w ^{(j)} - a_w ^{(j)}) _{j,w}
$$
where $j\in \{1,2\}$, and $w$ ranges over places of $\K$ lying above split $v \notin \Sigma$. Here the $a_w^{(j)}\in \mc{O}$ are lifts of $\bar{a}_w^{(j)}$, where
$X^2 + \bar{a} _w ^{(1)} X + \bar{a}_w ^{(2)}$ is the characteristic polynomial of $\bar{\rho}(\Frob _w)$. We will occasionally write $\bar{\rho}=\bar{\rho}_{\m}$.


\subsection{Essentially AIG representations}

Let $L/\Q_{\ell}$ be a finite extension, for some prime $\ell\neq p$. Eventually this $L$ will be one of the completions $F_v \simeq \K_{\tilde{v}}$, for $v \in \Sigma_0$. 
We say (after 3.2.1 in \cite{eh}) that a representation $V$ of $\GL_2(L)$, with coefficients in $E$ or $k_E$, is {\it{essentially AIG}} (which stands for absolutely irreducible and generic) if the conditions below are fulfilled -- we stress that {\it{generic}} is synonymous with {\it{infinite-dimensional}} here.

\begin{itemize}
\item[(1)] The $\GL_2(L)$-socle $\soc(V)$ is absolutely irreducible and generic.
\item[(2)] The quotient $V / \soc(V)$ contains no generic Jordan-Holder factors.
\item[(3)] The representation $V$ is the sum (or equivalently, the union) of its finite length submodules.
\end{itemize}

\noindent We remark that if $V$ is essentially AIG, then so is any non-zero subrepresentation $U \subset V$. Moreover, $\soc(U) = \soc(V)$. We will use this fact later on.

\medskip

\noindent We will use the term "essentially AIG" also for representations of $\prod _i \GL_2(L_i)$, for a finite collection of fields $L_i / \Q _{\ell_i}$ (possibly of varying residue characteristics $\ell_i$), by which we mean a tensor product $\otimes_i V_i$ of essentially AIG representations $V_i$ of each $\GL_2(L_i)$. The context below will be that of representations of $U(F_{\Sigma_0})$, which we always identify with $\prod_{v\in \Sigma_0}\GL_2(\K_{\tilde{v}})$ via the selection of places $\{\vv\}_{v \in \Sigma}$.

\subsection{The local Langlands correspondence in families}

\noindent One of the main theorems of \cite{eh} is a characterization of the local Langlands correspondence for $\GL(n)$ in families. The proof that it actually does exist (at least in the banal case, where $p$ is prime to the pro-order of the groups in question) is given in the recent work \cite{He}. Local Langlands in families puts the classical local Langlands correspondence in a family over a reduced complete Noetherian local $\mc{O}$-algebra $A$. Emerton and Helm give a short list of desiderata, which we verify for a candidate-representation occuring naturally in the course of the proof of our main result. For convenience, we recall the criteria here.

\begin{thm}[The local Langlands correspondence in families]
Let $A$ be a reduced complete Noetherian local $\mc{O}$-algebra.
Let $\Sigma _0$ be the set of places of $F$ from above. Suppose that for each $v \in \Sigma _0$ we are given a representation $\rho _{\vv} : \Gamma _{\K_{\vv}} \rightarrow \GL_2(A)$.  Then there is at most one (up to isomorphism) coadmissible smooth $A$-representation $\Pi _{\Sigma _0}=\pi(\{\rho_{\vv}\}_{v \in \Sigma_0})$ of $U(F_{\Sigma_0})\simeq \prod_{v\in \Sigma_0}\GL_2(\K_{\tilde{v}})$ such that:

\begin{itemize}
\item[(1)] $\Pi _{\Sigma _0}$ is cotorsion free over $A$ (that is, the smooth dual is torsion free -- which in turn means nonzero elements can only be annihilated by zero-divisors).
\item[(2)] $(\Pi _{\Sigma _0} / \varpi \Pi _{\Sigma _0})[\mathfrak{m}]$ is an essentially AIG representation (over $k=A/\m$).
\item[(3)] There is a Zariski dense subset of closed points $\mc{S} \subset \Spec A[1/p]$, such that for each point $\mathfrak{p} \in \mc{S}$, there exists an isomorphism:
$$ 
(\Pi _{\Sigma _0}\otimes_{\mc{O}} E)[\p]
\simeq \otimes _{v \in \Sigma _0} \pi(\rho_{\vv} \otimes_A \kappa (\mathfrak{p}))
$$
where $\kappa (\mf{p})$ is the fraction field of $A / \mf{p}$, and $\pi(\rho_{\vv} \otimes_A \kappa (\mathfrak{p}))$ is the representation of $\GL_2(\K_{\vv})$ associated to 
$\rho_{\vv}(\p):=\rho_{\vv} \otimes_A \kappa (\mathfrak{p})$ via the generic local Langlands correspondence (of Breuil-Schneider \cite{BS}).
\end{itemize}

\end{thm}

\begin{proof}
This is the content of Proposition 6.3.14 in \cite{eh}, except that we have dualized everything -- using Proposition C.5 and Lemma C.35 in \cite{em1}
\end{proof}

\noindent In this article we specialize the theorem to the following situation: We take $A=\T_{\Sigma,\m}$, for some maximal ideal $\m=\m_{\bar{\rho}} \subset \T_{\Sigma}$,
and the Galois input is $\rho_{\vv}=\rho_{\m,\vv}=\rho_{\m}|_{\Gamma_{\K_{\vv}}}$, where $\rho_{\m}: \Gamma_{\K}\rightarrow \GL_2(\T_{\Sigma,\m})$ is the universal modular $\Sigma$-deformation of $\bar{\rho}$ over the "big" Hecke algebra $\T_{\Sigma,\m}$. We write $\Pi _{\Sigma _0}$ for the resulting representation. For any $\mf{p} \in \Spec \mb{T} _{\Sigma, \mf{m}}$ we will write 
$$
\Pi _{\Sigma _0}(\mf{p}) = \kappa (\mf{p}) \otimes _{\mb{T} _{\Sigma, \mf{m}}} \Pi _{\Sigma _0}.
$$
Similarly, for $v \in \Sigma_p$, let $\Pi_v$ be the $\T_{\Sigma,\m}$-representation of $\GL_2(\K_{\vv})$ associated with $\rho_{\m}|_{\Gamma_{\K_{\vv}}}$
under the $p$-adic local Langlands correspondence for $\GL_2(\Q_p)$ (of Berger, Breuil, Colmez, Emerton, Kisin, Paskunas and others). Keep in mind that $\K_{\vv}=\Q_p$, so this is well-defined. Let $\Pi_{p}=\otimes_{v|p}\Pi_v$ -- viewed as a $\T_{\Sigma,\m}$-representation of $G(\Q_p)\simeq \prod_{v|p}\GL_2(\K_{\vv})$. As above, we may specialize,
$$
\Pi _p (\mf{p}) = \kappa (\mf{p}) \otimes _{\mb{T} _{\Sigma, \mf{m}}} \Pi _p.
$$
When $\p=\m$, we will write $\bar{\pi}_p$ instead of $\Pi_p(\m)=\Pi_p/\m \Pi_p$ (a smooth $G(\mb{Q}_p)$-representation over $k$). 


\section{The strategy of the proof of our main result}

The heart of the overall argument is to relate $\Pi_{\Sigma_0}$ to the module
$$
X := \Hom _{\mathbb{T}_{\Sigma, \mathfrak{m}}[G(\Q _p)]} ( \Pi _p, \widehat{H} ^0 _{\mc{O}, \Sigma, \m}),
$$
of $\mathbb{T} _{\Sigma, \mathfrak{m}}[G(\Q_p)]$-linear continuous homomorphisms $\Pi _p \rightarrow \widehat{H} ^0 _{\mc{O}, \Sigma, \m}$. Here
$\Pi _p$ is given the $\mathfrak{m}$-adic topology, and $\widehat{H} ^{0} _{\mc{O}, \Sigma, \m}$ is given the induced topology from $\widehat{H} ^{0} _{\mathcal{O}, \Sigma}$.
Note that $X$ is naturally a $\mathbb{T}_{\Sigma, \mathfrak{m}}$-module, with a smooth action of $U(F_{\Sigma_0})$, and there is a natural evaluation-map 
$ev_X: \Pi_p \overset{\curlywedge}{\otimes}_{\mathbb{T}_{\Sigma, \mathfrak{m}}}X \rightarrow  \widehat{H} ^0 _{\mc{O}, \Sigma, \m}$, which is equivariant for the $U(F_{\Sigma})$-action. The curly-wedge tensor product $\overset{\curlywedge}{\otimes}_{\mathbb{T}_{\Sigma, \mathfrak{m}}}$ is Emerton's notation (see Definition C.43 in \cite{em1}); it is the direct limit of $\varpi$-adically completed tensor products,
$$
\Pi_p \overset{\curlywedge}{\otimes}_{\mathbb{T}_{\Sigma, \mathfrak{m}}}X:= \varinjlim_{K_{\Sigma_0}}
\Pi_p \hat{\otimes}_{\mathbb{T}_{\Sigma, \mathfrak{m}}}X^{K_{\Sigma_0}},
$$
over compact open subgroups $K_{\Sigma_0}\subset U(F_{\Sigma_0})$.

\medskip

\noindent For every coadmissible $\mathbb{T}  _{\Sigma, \m} [U(F_{\Sigma _0})]$-submodule $Y\subset X$, we let $ev_Y$ denote the restriction of $ev_X$,
$$
ev_Y: \Pi _p \overset{\curlywedge}{\otimes} _{\mathbb{T}  _{\Sigma, \m}} Y \hookrightarrow  \Pi _p \overset{\curlywedge}{\otimes} _{\mathbb{T}  _{\Sigma, \m}} X \longrightarrow  \widehat{H} ^{0} _{\mathcal{O}, \Sigma, \m}.
$$
(The fact that the initial map is an embedding is Lemma C.48 in \cite{em1}.) In what follows, we will run into the map $ev_{Y,E}$ into $\widehat{H} ^{0} _{E, \Sigma, \m}$, obtained by tensoring $(-)\otimes_{\mc{O}}E$. For each prime $\p \subset \mathbb{T}  _{\Sigma, \m}[1/p]$, we take the $\p$-torsion and look at the induced map
$$
ev_{Y,E}[\p]: \Pi_p(\p)\otimes_{\kappa(\p)} (Y\otimes_{\mc{O}}E)[\p] \longrightarrow \widehat{H} ^{0} _{E, \Sigma, \m}[\p].
$$
On the other hand, we may first reduce $ev_Y$ modulo $\varpi$, and then take the $\m$-torsion. By Lemma C.45 in \cite{em1} this gives rise to a map
$$
ev_{Y,k}[\m]: \bar{\pi}_p\otimes_k (Y/\varpi Y)[\m] \longrightarrow H^{0} _{k, \Sigma, \m}[\m],
$$ 
where $H_{k,\Sigma}^0$ is the space of all $k$-valued $K^{\Sigma}$-invariant functions on $G(\Q)\backslash G(\A_f)$.

\medskip

\noindent In view of these definitions, our main result can be restated as saying there exists a coadmissible $\mathbb{T}  _{\Sigma, \m} [U(F_{\Sigma_0})]$-submodule $Y\subset X$ such that $ev _Y$ is an isomorphism and $Y \simeq \Pi _{\Sigma _0}$. To carry out this line of thought, we need two results whose proofs are given in 
next sections, and which form the technical core of the argument. We recall that $X_{ctf}\subset X$ denotes the maximal cotorsion-free coadmissible submodule of $X$.

\begin{prop}
Let $Y\subset X$ be a coadmissible $\mathbb{T}  _{\Sigma , \mf{m}} [U(F_{\Sigma_0})]$-module. Then the following two conditions are equivalent:

\begin{itemize}
\item[(1)] The map $ev_Y$ is an isomorphism.
\item[(2)] $Y$ is a faithful $\mathbb{T}  _{\Sigma, \mf{m}}$-module, and the map $ev_{Y,k}[\m]$ is injective.
\end{itemize}

\noindent Moreover, if $Y$ satisfies these conditions, then $Y = X_{ctf}$.
\end{prop}

\noindent We will combine this result with:

\begin{prop}
If $ev _{X _{ctf}}$ is an isomorphism, then $X _{ctf} \simeq \Pi _{\Sigma _0}$.
\end{prop}

\noindent Before proving our main Theorem, granting the above Propositions, let us mention one elementary lemma due to Emerton (Lemma 6.4.15, \cite{em1}). We record the proof for the sake of completeness.

\begin{lem}
Suppose that $\bar{\pi}_1$ and $\bar{\pi}_2$ are smooth representations of $G$ over $k$, $U$ is a vector space over $k$, and $f : \bar{\pi}_1 \otimes _k U \rightarrow \bar{\pi}_2$ is a $G$-equivariant $k$-linear map (the $G$-action on the source is defined via its action on $\bar{\pi}_1$). Assume that $\bar{\pi}_1$ is admissible and the $G$-socle $soc (\bar{\pi}_1)$ of $\bar{\pi}_1$ is multiplicity free. If for every non-zero element $u \in U$, the map $\bar{\pi}_1 \simeq \bar{\pi}_1 \otimes _k (k \cdot u) \rightarrow \bar{\pi}_2$ induced by $f$ is an embedding, then $f$ itself is an embedding.
\end{lem}

\begin{proof} Write $soc (\bar{\pi}_1) = \oplus _{i=1} ^s \bar{\pi} _{1,i}$ where $\bar{\pi} _{1,i}$ are pair-wise non-isomorphic irreducible admissible smooth $G$-representations. We have the isomorphism
$$\bigoplus _{i=1} ^s \bar{\pi}_{1,i} \otimes _k U \simeq soc (\bar{\pi}_1) \otimes _k U \simeq soc(\bar{\pi}_1 \otimes _k U)$$
If $f$ has a non-zero kernel, then this kernel has a non-zero socle, hence it has a non-zero intersection with $\bar{\pi}_{1,i} \otimes _k U$ for at least one $i$. Thus it contains $\bar{\pi}_{1,i} \otimes _k (k\cdot u)$ for some non-zero $u \in U$. But this means that the map $\bar{\pi}_1 \otimes _k (k \cdot u) \rightarrow \bar{\pi}_2$ is not injective, which contradicts our assumption.
\end{proof}

\begin{proof}[Proof of Theorem 1]
By Proposition 2 it is enough to show that $ev _{X_{ctf}}$ is an isomorphism. By Proposition 1, this is equivalent to $ev _{X _{ctf},k}[\mf{m}]$ being injective, and $X_{ctf}$ being faithful over $\mathbb{T}  _{\Sigma , \mf{m}}$. First, in regards to the faithfulness: We will show later that $ev_{X,E}$ is certainly onto (this is part of Proposition 4 below). In particular, $X$ is faithful. Now Proposition C.40 in \cite{em1} tells us that so is $X_{ctf}$. What remains is the injectivity of the evaluation map,
$$
\bar{\pi}_p\otimes_k (X_{ctf}/\varpi X_{ctf})[\m] \longrightarrow H^{0} _{k, \Sigma, \m}[\m].
$$
By Lemma 1 it suffices to show, for every nonzero $u \in (X_{ctf}/\varpi X_{ctf})[\m]$, that the induced map $\bar{\pi}_p \rightarrow H^{0} _{k, \Sigma, \m}[\m]$ is injective.
However, each such $G(\Q_p)$-equivariant map is either trivial or injective -- by our standing hypothesis on $\bar{\rho}$. Namely, by known properties of mod $p$ local Langlands for $\GL_2(\Q_p)$, we know that $\bar{\pi}_p$ is irreducible since each restriction $\bar{\rho}|_{\Gamma_{\K_{\vv}}}$ is assumed to be irreducible for all $v|p$.
\end{proof}


\section{On coadmissible Hecke modules}

\subsection{Allowable points}

\noindent Following standard terminology in the subject, we define {\it{allowable points}} to be those prime ideals $\mathfrak{p} \in \Spec \mathbb{T}  _{\Sigma, \m} [1/p]$ 
for which the specialization $\rho_{\m}(\p)$ is crystalline above $p$, with distinct Hodge-Tate weights ("regular").

\medskip

\noindent We record the following lemma for later use.

\begin{lem}
If $\mathfrak{p} \in \Spec \mathbb{T}  _{\Sigma, \m} [1/p]$ is an allowable point, then there is a $U(F_{\Sigma_0})$-equivariant $\kappa (\mathfrak{p})$-linear isomorphism
$$
(X\otimes_{\mc{O}}E)[\p]\simeq \otimes_{v\in \Sigma_0} \pi(\rho_{\m,\vv}(\p))=: \pi_{\Sigma_0}(\rho_{\m}(\p)).
$$
Furthermore, the evaluation map induces an isomorphism
$$
ev_{X,E}[\p]: \Pi_p(\p)\otimes_{\kappa(\p)} (X\otimes_{\mc{O}}E)[\p] \overset{\sim}{\longrightarrow} (\widehat{H} ^{0} _{E, \Sigma, \m}[\p]_{loc.alg.})^-=: M(\p)_E.
$$
(The closure of the space of locally algebraic vectors in $\widehat{H} ^{0} _{E, \Sigma, \m}[\p]$.)
\end{lem}

\begin{proof}
Since $\rho_{\m,\vv}(\p)$ is crystalline and regular, a fundamental result of Berger and Breuil tells us that $\Pi_v(\p)$ is the universal unitary completion of
the locally algebraic representation $BS(\rho_{\m,\vv}(\p))$ associated with $\rho_{\m,\vv}(\p)$ by Breuil and Schneider. A fortiori, $\Pi_p(\p)=\hat{\otimes}_{v|p}\Pi_v(\p)$ 
is the universal unitary completion of $BS(\p):=\otimes_{v|p} BS(\rho_{\m,\vv}(\p))$ -- a locally algebraic representation of $G(\Q_p)$ over $\kappa(\p)$, which we factor as $\xi(\p)\otimes_{\kappa(\p)} \pi(\p)$. Here the algebraic part $\xi(\p)$ records the Hodge-Tate weights, and the smooth part $\pi(\p)$ corresponds to the Weil-Deligne representations
$WD(\rho_{\m,\vv}(\p))$ under the classical local Langlands correspondence.

\medskip

\noindent The locally algebraic vectors in $\widehat{H} ^{0} _{E, \Sigma, \m}[\p]$ have a description in terms of automorphic representations $\pi$ of $G$,
$$
\widehat{H} ^{0} _{E, \Sigma, \m}[\p]_{loc.alg.}={\bigoplus}_{\pi: \pi_{\infty}=\xi(\p)} BS(\p) \otimes \pi_{\Sigma_0} \otimes (\pi_f^{\Sigma})^{K^{\Sigma}}[\p].
$$
Here we have used multiplicity one ($m_{\pi}=1$) for $G$, and local-global compatibility at $p$ (due to Caraiani and others). There is a unique $\pi$ contributing to the right-hand side -- its base change to $\GL_2(\A_{\K})$ is uniquely determined, being cuspidal and having a prescibed Hecke eigensystem away from $\Sigma$. (Base change is "injective" at the inert places where $\pi$ is unramified.) Taking the closure, 
$$
M(\p)_E\simeq \Pi_p(\p)\otimes_E \pi_{\Sigma_0}(\rho_{\m}(\p)).
$$
We have used that $\pi_{\Sigma_0}$ is generic (so local Langlands and generic local Langlands coincide), local-global compatibility holds at the places in $\Sigma_0$, and $\pi_f^{\Sigma}$ is unramified.

\medskip

\noindent Hence, again using the fact that $\Pi _p(\p)$ is the completion of the locally algebraic representation $BS(\p)$ -- which is irreducible (by a result of D. Prasad -- see Theorem 1, part 2, p. 126 in his Appendix to \cite{ST}), we have
$$\Hom _{E[G(\Q _p)]} (\Pi _p(\p), \widehat{H}^0 _{E,\Sigma, \m}[\p])=$$
$$\Hom _{E[G(\Q _p)]} (BS(\p), \widehat{H}^0 _{E,\Sigma, \m}[\p])=$$
$$\Hom _{E[G(\Q _p)]} (BS(\p), \widehat{H}^0 _{E,\Sigma, \m}[\p] _{loc.alg.})=$$
$$\Hom _{E[G(\Q _p)]} (BS(\p), M(\p)_E)=$$
$$\Hom _{E[G(\Q _p)]} (\Pi _p(\p), M(\p)_E)$$

\noindent This string of equalities, together with the natural identification
$$
(X \otimes _{\mathcal{O}} E)[\mathfrak{p}] \simeq \Hom _{E[G(\Q _p)} (\Pi _p(\p), \widehat{H}^0 _{E,\Sigma, \m}[\p]),
$$
implies
$$
(X \otimes _{\mathcal{O}} E)[\mathfrak{p}] \simeq \Hom _{E[G(\Q _p)} (\Pi _p(\p), M(\p)_E),
$$
from which both statements of the Lemma are easily deduced.
\end{proof}

\noindent Of course, the way this will be used below, is by showing that the allowable points $\mc{S}_{al.}$ are Zariski dense -- so that we may take $\mc{S}=\mc{S}_{al.}$ in part (3) of Theorem 2 which characterizes $\Pi_{\Sigma_0}$. The Lemma identifies $(X \otimes _{\mathcal{O}} E)[\mathfrak{p}]$ with $\Pi_{\Sigma_0}\otimes_{ \mathbb{T}  _{\Sigma, \m}}\kappa(\p)$, for $\p \in \mc{S}_{al.}$. We will eventually use this in the proof of Proposition 2. 

\subsection{Proof of Proposition 1 -- following Emerton}

\noindent Recall that a submodule $Y$ of a $\mathcal{O}$-torsion free module $X$ is {\it{saturated}} if $X/Y$ is $\mathcal{O}$-torsion free (see Definition C.6 in \cite{em1}). The proof of the following Proposition is based on that of Proposition 6.4.2, p. 82, in \cite{em1}. Here, for convenience, we will walk the reader through the relevant facts from Emerton's comprehensive Appendix C on coadmissible modules.  

\begin{prop}
If $Y\subset X$ is a saturated coadmissible $\mathbb{T}  _{\Sigma, \m} [U(F_{\Sigma_0})]$-submodule, then the following conditions are equivalent:

\begin{itemize}
\item[(a)] $Y$ is a faithful $\mathbb{T}  _{\Sigma, \m}$-module.
\item[(b)] For each level $K_{\Sigma _0} \subset U(F_{\Sigma_0})$, the submodule of invariants $Y ^{K_{\Sigma _0}}$ is a faithful $\mathbb{T} _{\Sigma}(K _{\Sigma _0})_{\m}$-module.
\item[(c)] For each classical allowable closed point $\mathfrak{p} \in \Spec \mathbb{T}  _{\Sigma, \m} [1/p]$, the inclusion 
$(Y \otimes _{\mathcal{O}} E)[\mathfrak{p}] \subset (X \otimes _{\mathcal{O}} E)[\mathfrak{p}]$ is an equality.
\item[(d)] For each level $K_{\Sigma _0} \subset G_{\Sigma _0}$, and for each classical allowable closed point $\mathfrak{p} \in \Spec \mathbb{T}  _{\Sigma, \m} [1/p]$, the inclusion $(Y ^{K_{\Sigma _0}}\otimes_{\mc{O}}E) [\mathfrak{p}] \subset (X ^{K _{\Sigma _0}}\otimes_{\mc{O}}E)[\mathfrak{p}]$ is an equality.
\item[(e)] $X_{ctf} \subset Y$.
\end{itemize}
\end{prop}

\begin{proof}
As a $\mb{T} _{\Sigma, \m}[U(F_{\Sigma_0})]$-module, $X$ is cofinitely generated and coadmissible (as explained in \cite{cs}, subsection 4.9). Hence, by Corollary C.34 from \cite{em1}, since $Y$ is saturated in $X$, we infer that $Y$ is cofinitely generated over $\mb{T} _{\Sigma, \m}$.

\medskip

\noindent Observe that clearly (c) $\Leftrightarrow$ (d), and (b) $\Rightarrow$ (a).

\medskip

\noindent (a) $\Leftrightarrow$ (c): Lemma 2 above shows that, for any allowable point $\p \in \Spec \mb{T} _{\Sigma, \m}[1/p]$, the $U(F_{\Sigma_0})$-representation $(X \otimes _{\mc{O}} E) [\p]$ is in particular irreducible. Hence, for any allowable $\p$, we have $(Y \otimes _{\mc{O}} E)[\p] \not = 0$ if and only if $(Y \otimes _{\mc{O}} E)[\p] = (X \otimes _{\mc{O}} E)[\p]$. We know from our previous work that the allowable points are Zariski dense in $\Spec \mb{T} _{\Sigma, \m}[1/p]$ (see \cite{cs}, subsection 4.6), and since $\mb{T} _{\Sigma, \m}$ is reduced, we conclude from Proposition C.36 of \cite{em1}, that (a) and (c) are equivalent.

\medskip

\noindent (d) $\Rightarrow$ (b): By the irreducibility remarks just made, $(Y^{K_{\Sigma _0}}\otimes_{\mc{O}}E)[\p] \not = 0$ for each allowable point $\p \in \Spec \mb{T}_{\Sigma} (K_{\Sigma _0}) _{\m}$. Referring to Proposition C.22 in \cite{em1}, $\p$ therefore lies in the cosupport of $Y^{K_{\Sigma _0}}$ -- which by the preceding remarks in loc. cit. means 
$\p$ contains the annihilator ideal $\Ann_{ \mb{T}_{\Sigma} (K_{\Sigma _0}) _{\m}}(Y^{K_{\Sigma _0}})$. By the aforementioned density, $\cap \p=0$. Consequently, the annihilator ideal is  trivial.

\medskip

\noindent Altogether this shows the first four conditions are equivalent. We end by showing they are all equivalent to (e).
Suppose we can prove that $X_{ctf}$ can be characterized as the unique saturated, coadmissible $\mb{T}_{\Sigma, \m}[U(F_{\Sigma_0})]$-submodule of $X$ which is both faithful and cotorsion free over $\mb{T}_{\Sigma, \m}$. Note that, by Proposition C.40 of \cite{em1}, $Y_{ctf}$ is faithful if $Y$ is. As a result, admitting the posited characterization of $X_{ctf}$, we find that (a) implies $Y_{ctf} = X_{ctf}$ -- and, in particular condition (e) that $X_{ctf} \subset Y$. In conclusion, (a) $\Rightarrow$ (e).
For the converse, observe that the claim implies $X_{ctf}$ satisfies (a), the faithfulness. Thus, if $X_{ctf} \subset Y$, then $Y$ automatically satisfies (a) as well. To summarize, 
(a) and (e) are equivalent.

\medskip

\noindent We are left with proving the characterization of $X_{ctf}$. As $X$ trivially satisfies (c), it also satisfies (a), and by Proposition C.40 of \cite{em1}, we get that $X_{ctf}$ satisfies (a). Hence $X_{ctf}$ {\it{is}} a faithful $\mb{T}_{\Sigma, \m}$-cotorsion free $\mb{T}_{\Sigma, \m}[U(F_{\Sigma_0})]$-submodule of $X$. Let $Y$ be any saturated, coadmissible $\mb{T}_{\Sigma, \m}[U(F_{\Sigma_0})]$-submodule of $X$ which is faithful and cotorsion free over $\mb{T}_{\Sigma, \m}$. Obviously $Y \subset X_{ctf}$. Moreover, since we know (a) implies (c), we get an equality $(Y \otimes _{\mc{O}} E)[\p] =(X_{ctf} \otimes _{\mc{O}}E)[\p]$ for allowable points $\p$ -- which are Zariski dense in $\Spec \mb{T} _{\Sigma, \m}$. Thus Proposition C.41 of \cite{em1} implies that $Y=X_{ctf}$, as desired.
\end{proof}

\noindent This is one of the key technical ingredients which goes into the proof of Proposition 1. The other is a variant of "Zariski density of crystalline points":

\begin{prop}
Let $Y\subset X$ be a saturated coadmissible $\mathbb{T}  _{\Sigma, \m} [U(F_{\Sigma_0})]$-submodule, which satisfies the equivalent conditions (a)--(e) of the previous Proposition. Then $ev _{Y,E}$ is surjective. Equivalently, for each sufficiently small level $K_{\Sigma _0}$, the map 
$$
ev _{Y,E}(K_{\Sigma _0}): \Pi _p \hat{\otimes} _{\mathbb{T}  _{\Sigma, \m}} (Y^{K_{\Sigma_0}}\otimes_{\mc{O}}E) \longrightarrow \hat{H}^0(K_{\Sigma_0}K^{\Sigma})_{E,\m}
$$
is surjective.
\end{prop}

\begin{proof}
The last  "equivalence" of the Proposition follows by taking $K_{\Sigma _0}$-invariants or, conversely, passing to the limit over all levels $K_{\Sigma _0}$ -- using the very definition of $\overset{\curlywedge}{\otimes}_{\mathbb{T}  _{\Sigma, \m}}$.

\medskip

\noindent We first prove that the image of $ev _{Y,E}$ contains $\oplus _{\p \in \mc{C}} \widehat{H}^0 _{E, \Sigma, \m}[\p] _{loc.alg.}$, where $\mc{C}$ is the set of classical allowable points $\p \in \Spec \mb{T} _{\Sigma, \m}[1/p]$. Proposition 3, part (c), gives us the equality $(Y \otimes _{\mc{O}} E)[\p] = (X \otimes _{\mc{O}} E)[\p]$. Therefore,
$$
\im(ev_{Y,E}[\p])=\im(ev_{X,E}[\p])=M(\p)_E\supset \widehat{H}^0 _{E, \Sigma, \m}[\p] _{loc.alg.},
$$
by the second half of Lemma 2. Letting $\p$ vary, we deduce our first claim that $\im(ev _{Y,E})$ contains $\oplus _{\p \in \mc{C}} \widehat{H}^0 _{E, \Sigma, \m}[\p] _{loc.alg.}$. However, the latter is {\it{dense}} in $\widehat{H}^0 _{E,\Sigma,\m}$, by "Zariski density of crystalline points" (see Proposition 3 in subsection 4.6 of \cite{cs}, for example). Thus, a priori, $\im(ev _{Y,E})$ is at least dense in $\widehat{H}^0 _{E,\Sigma,\m}$. However, $\im(ev _{Y,E})$ is in fact closed, as we now explain.

\medskip

\noindent Using Lemma 3.1.16 on p. 17 in \cite{em1}, we first note that $\Pi _p \widehat{\otimes} _{\mb{T} _{\Sigma, \m}} Y^{K_{\Sigma_0}}$ is a $\varpi$-adically admissible 
representation of $G(\Q _p)$ over $\mb{T} _{\Sigma, \m}$ (acting through the first factor of the tensor product). Admissibility of $\hat{H}^0(K_{\Sigma_0}K^{\Sigma})_{\mc{O},\m}$ is immediate (by finiteness of the class number of $G$). We may then apply Proposition 3.1.3 of \cite{em1} to the $G(\Q_p)$-map $ev_Y(K_{\Sigma_0})$ between admissible representations. It states precisely that the induced map on Banach $E$-spaces, $ev _{Y,E}(K_{\Sigma _0})$, has {\it{closed}} image. We conclude that each  $ev _{Y,E}(K_{\Sigma _0})$ must in fact be onto.
\end{proof}


\noindent Having established these preliminary results, we can finally proceed to the proof of Proposition 1. (We remark that the proof is almost identical to that of Emerton's Theorem 6.4.9, p. 85, in \cite{em1}.)

\begin{proof}[Proof of Proposition 1]
Suppose first that $ev_Y$ is an isomorphism. Since $\widehat{H} ^{0} _{\mathcal{O},\Sigma, \m}$ is a faithful $\mathbb{T}  _{\Sigma, \m}$-module, we see that $Y$ must be a faithful $\mathbb{T}  _{\Sigma, \m}$-module. Reducing the map $ev_Y$ modulo $\varpi$, and taking $\mathfrak{m}$-torsion, we find that $ev_{Y,k}[\mathfrak{m}]$ is at least injective. (It may not be onto, a priori.) This shows that (1) $\Rightarrow$ (2).

\medskip

\noindent (2) $\Rightarrow$ (1): Conversely, if $ev_{Y,k} [\mathfrak{m}]$ is injective, then by Lemma C.46 of \cite{em1}, we see that $ev_Y$ is injective, with saturated image. Lemma C.52 of \cite{em1} then implies that $Y$ must be saturated in $X$. If $Y$ is furthermore faithful as a $\mathbb{T}  _{\Sigma, \m}$-module, then Proposition 4 shows that $ev_{Y,E}$ is surjective. It follows that $ev_Y$ is in fact an isomorphism -- again, because $\im(ev_Y)$ is saturated in $\hat{H}_{\mc{O},\Sigma,\m}^0$.

\medskip

\noindent Finally, assume $ev_Y$ satisfies these conditions (1) and (2). We must show $Y=X_{ctf}$. As we have already noted in the previous paragraph, it follows that $Y$ is saturated in $X$, and so (e) of Proposition 3 shows one inclusion, $X_{ctf} \subset Y$. Since $X_{ctf}$ is saturated in $X$, and therefore also in $Y$, we conclude that the induced map
$$
(X_{ctf}/\varpi X_{ctf})[\mathfrak{m}] \longrightarrow (Y/\varpi Y)[\mathfrak{m}]
$$
is an embedding (even before taking $\m$-torsion), and thus that $ev _{X_{ctf},k} [\mathfrak{m}]$ is injective, since $ev_{Y,k} [\mathfrak{m}]$ is. The equivalence (1) $\Leftrightarrow $ (2) applied to $X_{ctf}$ then tells us that $ev_{X_{ctf}}$ is an isomorphism. In other words, the inclusion $X_{ctf} \subset Y$ induces an isomorphism
$$
 \Pi _p  \overset{\curlywedge}{\otimes} _{\mathbb{T}  _{\Sigma, \m}} X _{ctf} \overset{\sim}{\longrightarrow}  \Pi _p  \overset{\curlywedge}{\otimes} _{\mathbb{T}  _{\Sigma, \m}} Y.
 $$
It follows from Lemma C.51 of \cite{em1} that the inclusion $X_{ctf} \subset Y$ is an equality as required.
\end{proof}

\subsection{Proof of Proposition 2 -- Ihara's lemma for $U(2)$}

Let $W^{gl}(\bar{\rho})$ be the set of global Serre weights of $\bar{\rho}$. By definition, those are the finite-dimensional irreducible $k[G(\mb{Z}_p)]$-representations $V$ for which $$
\Hom _{k[G(\mathbb{Z} _p)]}(V, H^{0} _{k, \Sigma, \m}[\mathfrak{m}])\neq 0.
$$ 
It is easy to see that this space can be identified with $S_{V}(G(\Z_p)K^{\Sigma},k)_{\m}[\m]$, where $S_{V}(G(\Z_p)K^{\Sigma},k)$ is the space of mod $p$ algebraic modular forms of weight $V$, and level $G(\Z_p)K^{\Sigma}$. More concretely,
$$
S_{V}(G(\Z_p)K^{\Sigma},k)=\varinjlim_{K_{\Sigma_0}}S_{V}(G(\Z_p)K_{\Sigma_0}K^{\Sigma},k),
$$
where $S_{V}(G(\Z_p)K_{\Sigma_0}K^{\Sigma},k)$ is the (finite-dimensional) space of functions
$$
f: G(\Q)\backslash G(\A_f)/K_{\Sigma_0}K^{\Sigma} \rightarrow V^{\vee}, \y f(gu)=u^{-1}f(g), \y  u \in G(\Z_p).
$$
The Buzzard-Diamond-Jarvis conjecture (proved by Gee et al. in the case of unitary groups in two variables) gives an explicit description of $W^{gl}(\bar{\rho})$ in terms of the restrictions to inertia, $\{\bar{\rho}|_{I_{\K_{\vv}}}\}_{v|p}$. Recall from the introduction that $\bar{\rho}$ is assumed to occur at {\it{maximal}} level $K_{\Sigma_0}$, with weight
$V$ in the {\it{Fontaine-Laffaille}} range (which means the highest weights of $V$ lie in $[0,p-2)$). Once and for all, choose such a $V \in W^{gl}(\bar{\rho})$.
One of the goals of \cite{sor2} was to prove the following:

\begin{prop}
$S_{V}(G(\Z_p)K^{\Sigma},k)_{\m}[\m]$ is essentially AIG.
\end{prop}

\begin{proof}
Up to notational inconsistency, this is Theorem 4, p. 28 in \cite{sor2} -- which in fact gives a $U(n)$-analogue contingent on Ihara's lemma, which is known (even trivial) for $n=2$.
\end{proof}
 
\begin{cor}
$\Hom _{k[G(\Q_p)]}( \bar{\pi}_p, H ^{0} _{k, \Sigma, \m} [\mathfrak{m}])$ is essentially AIG.
\end{cor}

\begin{proof} Since it is known that $W^{gl}(\bar{\rho})=\soc_{k[G(\Z_p)]}(\bar{\pi}_p)$, see Theorem A in \cite{GLS},
we have an embedding $V \hookrightarrow \bar{\pi}_p|_{G(\Z_p)}$, whose image generates $\bar{\pi}_p$ over $G(\Q_p)$, because $\bar{\pi}_p\leftrightarrow \{\bar{\rho}|_{\Gamma_{\K_{\vv}}}\}_{v|p}$ is irreducible -- by known properties of mod $p$ local Langlands for $\GL_2(\Q_p)$. We infer that the restriction map
$$
\Hom _{k[G(\Q_p)]} ( \bar{\pi}_p, H ^{0} _{k, \Sigma, \m} [\mathfrak{m}]) \longrightarrow \Hom _{k[G(\mathbb{Z} _p)]}(V, H^{0} _{k, \Sigma, \m}[\mathfrak{m}])
$$
is injective, and the result follows from Proposition 5 and a remark in 2.2.
\end{proof}


\noindent We need one more preliminary result in order to show $X_{ctf}$ interpolates the local Langlands correspondence. (Analogous to Lemma 6.4.5, p. 83, in \cite{em1}, except that we only need allowable {\it{prime}} ideals.) 

\begin{lem}
The closure (in the sense of definition C.28 of \cite{em1}) of 
$$
\cup  _{\textrm{$\p$ allowable}} X[\p]
$$
in $X$ coincides with $X _{ctf}$. 
\end{lem}

\begin{proof}
We start off with a general remark, which will be used for both inclusions. Let $Y\subset X$ be an arbitrary saturated, coadmissible $\T_{\Sigma,\m}[U(F_{\Sigma_0})]$-submodule, which is a faithful $\T_{\Sigma,\m}$-module. By Proposition (3), part (c), we know that $Y[\p]=X[\p]$ for all allowable $\p$. (A priori, we only know this after tensoring with $E$, but $Y$ is saturated -- see also C.41 in \cite{em1}.) In particular, any such $Y$ contains the union $\cup  _{\textrm{$\p$ allowable}} X[\p]$, and therefore its closure $W$ since $Y$ is closed.
(By C.31 in loc. cit., this is equivalent to $Y$ being coadmissible.)

\medskip

\noindent ($\subset$) Take $Y=X_{ctf}$ in the previous paragraph.

\medskip

\noindent ($\supset$)  To show that $W \supset X_{ctf}$, first observe that $W$ is saturated (C.30), coadmissible (C.31), and {\it{faithful}}: If $t \in \T_{\Sigma,\m}$ acts trivially in $W$, it acts trivially on each $X[\p]$ -- in other words, $t \in \Ann_{\T_{\Sigma,\m}}X[\p]=\p$. By Zariski density, $\cap \p=0$, so that $t=0$. From the first paragraph of the proof, we conclude that $W$ is the {\it{minimal}} saturated coadmissible faithful submodule of $X$. Since $W_{ctf}\subset W$ has all these three properties, $W=W_{ctf}$ must be cotorsion-free. As we have already seen in the proof of Propsition 3, part (e), this implies $W=X_{ctf}$.
\end{proof}

\begin{proof}[Proof of Proposition 2]
Let $\mc{S}$ denote the set of closed points $\mathfrak{p} \in \Spec \mathbb{T}  _{\Sigma, \m}$, which are allowable. As noted above, $\mc{S}$ is Zariski dense. We want to show that $X _{ctf}$, $\mc{S}$ and $A = \mb{T} _{\Sigma, \mf{m}}$ satisfy the three conditions determining the local Langlands correspondence in families (see Section 2.3).

\begin{itemize}
\item[(1)] $X _{ctf}$ is cotorsion-free over $\mb{T} _{\Sigma, \mf{m}}$ by definition.
\item[(2)] $(X_{ctf}/\varpi X_{ctf})[\mathfrak{m}]$ is essentially AIG, since (by saturation) it sits inside
$$
(X/\varpi X)[\mathfrak{m}]\simeq \Hom _{k[G(\Q_p)]}( \bar{\pi}_p, H ^{0} _{k, \Sigma, \m} [\mathfrak{m}]),
$$
which was just shown to be essentially AIG.
\item[(3)] By Lemma 2 in section 4.1, for all $\p  \in \mc{S}$,
$$
(X\otimes_{\mc{O}}E)[\p]\simeq \otimes_{v\in \Sigma_0} \pi(\rho_{\m,\vv}(\p)).
$$
On the left-hand side, the previous lemma allows us to replace $X$ by $X_{ctf}$.
\end{itemize}

\end{proof}


\noindent {\it{Przemyslaw Chojecki}}

\noindent {\sc{Institut Mathematique de Jussieu, Paris, France.}}

\noindent {\it{E-mail address}}: {\texttt{chojecki@math.jussieu.fr}}

\bigskip

\noindent {\it{Claus Sorensen}}

\noindent {\sc{Department of Mathematics, UCSD, La Jolla, CA, USA.}}

\noindent {\it{E-mail address}}: {\texttt{csorensen@ucsd.edu}}


\begin{thebibliography}{EGH}


\bibitem[BS]{BS} C. Breuil and P. Schneider, {\it{First steps towards p-adic Langlands functoriality}}. J. Reine Angew. Math. 610, 2007, 149-180.


\bibitem[CHT]{CHT} Clozel, Harris, Taylor, {\it{Automorphy for some l-adic lifts of automorphic mod l Galois representations}}.
With Appendix A, summarizing unpublished work of Russ Mann, and Appendix B by Marie-France Vigneras.
Publ. Math. Inst. Hautes Etudes Sci. No. 108 (2008), 1Ð181. 

\bibitem[CS]{cs} P. Chojecki and C. Sorensen, {\it{Weak local-global compatibility in the $p$-adic Langlands program for $U(2)$}}. Preprint (2013).


\bibitem[Em1]{em1} M. Emerton, {\it{Local-global compatibility in the p-adic Langlands programme for $\GL_{2 / \mathbb{Q}}$}}. Preprint (2011).




\bibitem[EH]{eh} M. Emerton, D. Helm, {\it{The local Langlands correspondence for $\GL _n$ in families}}. Preprint (2011).

\bibitem[GLS]{GLS} Gee, Liu, Savitt, {\it{Crystalline extensions and the weight part of Serre's conjecture}}. Algebra Number Theory 6 (2012), no. 7, 1537Ð1559.

\bibitem[He]{He} D. Helm, {\it{Whittaker models and the integral Bernstein center for $\GL_n(F)$}}. Preprint (2013). 



\bibitem[Sor2]{sor2} C. Sorensen, {\it{The local Langlands correspondence in families and Ihara's lemma for U(n)}}. Preprint (2014).

\bibitem[ST]{ST} Schneider, Teitelbaum, {\it{$U(g)$-finite locally analytic representations}}. With an appendix by Dipendra Prasad. Represent. Theory 5 (2001), 111Ð128.

\end{thebibliography}
\end{document}